\pdfoutput=1
\RequirePackage{ifpdf}
\ifpdf 
\documentclass[pdftex]{sigma}
\else
\documentclass{sigma}
\fi

\numberwithin{equation}{section}

\newtheorem{Theorem}{Theorem}[section]
\newtheorem{Corollary}[Theorem]{Corollary}
\newtheorem{Lemma}[Theorem]{Lemma}
\newtheorem{Proposition}[Theorem]{Proposition}
 { \theoremstyle{definition}
\newtheorem{Definition}[Theorem]{Definition}
\newtheorem{Note}[Theorem]{Note}
}

\begin{document}

\newcommand{\arXivNumber}{1801.06083}

\renewcommand{\thefootnote}{}

\renewcommand{\PaperNumber}{044}

\FirstPageHeading

\ShortArticleName{The $q$-Onsager Algebra and the Universal Askey--Wilson Algebra}

\ArticleName{The $\boldsymbol{q}$-Onsager Algebra \\ and the Universal Askey--Wilson Algebra\footnote{This paper is a~contribution to the Special Issue on Orthogonal Polynomials, Special Functions and Applications (OPSFA14). The full collection is available at \href{https://www.emis.de/journals/SIGMA/OPSFA2017.html}{https://www.emis.de/journals/SIGMA/OPSFA2017.html}}}

\Author{Paul TERWILLIGER}

\AuthorNameForHeading{P.~Terwilliger}

\Address{Department of Mathematics, University of Wisconsin, Madison, WI 53706-1388, USA}
\Email{\href{mailto:terwilli@math.wisc.edu}{terwilli@math.wisc.edu}}

\ArticleDates{Received January 25, 2018, in final form May 01, 2018; Published online May 07, 2018}

\Abstract{Recently Pascal Baseilhac and Stefan Kolb obtained a PBW basis for the $q$-Onsager algebra $\mathcal O_q$. They defined the PBW basis elements recursively, and it is obscure how to express them in closed form. To mitigate the difficulty, we bring in the universal Askey--Wilson algebra $\Delta_q$. There is a natural algebra homomorphism $\natural \colon \mathcal O_q \to \Delta_q$. We apply $\natural $ to the above PBW basis, and express the images in closed form. Our results make heavy use of the Chebyshev polynomials of the second kind.}

\Keywords{$q$-Onsager algebra; universal Askey--Wilson algebra; Chebyshev polynomial}

\Classification{33D80; 17B40}

\renewcommand{\thefootnote}{\arabic{footnote}}
\setcounter{footnote}{0}

\section{Introduction}

In the 1944 paper~\cite{Onsager} Lars Onsager obtained the free energy of the two-dimensional Ising model in a zero magnetic field. In that paper an infinite-dimensional Lie algebra was introduced; this algebra is now called the Onsager algebra and denoted by~$\mathcal O$. Onsager defined his algebra by giving a linear basis and the action of the Lie bracket on the basis. In~\cite{perk} Perk gave a presentation of $\mathcal O$ by generators and relations. This presentation involves two generators and two relations, called the Dolan/Grady relations~\cite{Dolgra}. This presentation is discussed in \cite[Remark~9.1]{madrid}. Via this presentation, the universal enveloping algebra of~$\mathcal O$ admits a~$q$-deformation~$\mathcal O_q$ called the $q$-Onsager algebra \cite{bas1,qSerre}. The algebra~$\mathcal O_q$ is associative and infinite-dimensional. It is defined by two generators and two relations called the $q$-Dolan/Grady relations; these are given in (\ref{eq:dg1}), (\ref{eq:dg2}) below. The $q$-Dolan/Grady relations first appeared in algebraic combinatorics, in the study of $Q$-polynomial distance-regular graphs \cite[Lemma~5.4]{tersub3}. Shortly thereafter they appeared in physics, in the study of statistical mechanical models \cite[Section~2]{bas1}. Up to the present, the representation theory of $\mathcal O_q$ remains an active area of research in mathematics \cite{INT, ItoTer, IT:aug, LS99, qSerre, madrid, uaw, pospart, aw} and physics \cite{bas2,bas1, bas6, basDef,basbel, bas8, basXXZ, bas4, BK05, basKoi, basnc, BVu}. This theory involves a linear algebraic object called a tridiagonal pair~\cite{TD00}. A finite-dimensional irreducible~$\mathcal O_q$-module is essentially the same thing as a tridiagonal pair of $q$-Racah type \cite[Theorem~3.10]{qSerre}. These tridiagonal pairs are classified up to isomorphism in \cite[Theorem~3.3]{ItoTer}. In \cite[Theorem~2.1]{IT:aug}, Ito and the present author gave a linear basis for~$\mathcal O_q$, called the zigzag basis. More information about this basis can be found in \cite[Note~4.7]{pospart}. In \cite{basbel}, Baseilhac and Belliard conjectured another linear basis for~$\mathcal O_q$; this one is motivated by how $\mathcal O_q$ is related to the reflection equation algebra~\cite{BK05, basnc}. In~\cite{BK}, Baseilhac and Kolb introduced two automor\-phisms~$T_0$,~$T_1$ of $\mathcal O_q$ that are roughly analogous to the Lusztig automorphisms of~$U_q(\widehat{\mathfrak{sl}}_2)$. They used~$T_0$,~$T_1$ and a method of Damiani~\cite{damiani} to obtain a~Poincar\'e--Birkhoff--Witt (or PBW) basis for $\mathcal O_q$ \cite[Theorem~4.3]{BK}. In our view this PBW basis is important and worthy of further study. In the present paper we study the following aspect. In \cite[Section~3.1]{BK} the PBW basis is defined recursively, and it is obscure how to express it in closed form. In order to mitigate the difficulty, we bring in a related algebra which we now describe. In~\cite{Z} Zhedanov introduced the Askey--Wilson algebra AW(3) and used it to describe the Askey--Wilson polynomials. In~\cite{uaw} the present author introduced a central extension of~AW(3), called the universal Askey--Wilson algebra $\Delta_q$. In~\cite{Huang2}, Hau-Wen Huang classified up to isomorphism the finite-dimensional irreducible $\Delta_q$-modules for~$q$ not a root of unity. A~linear basis for $\Delta_q$ is given in \cite[Theorem~7.5]{uaw}. There is a natural algebra homomorphism $\natural\colon \mathcal O_q \to \Delta_q$ \cite[Definition~10.4]{uaw}; this is described below~(\ref{eq:rel4}) in the present paper. We use $\natural $ to describe the PBW basis for $\mathcal O_q$ in the following way. We apply $\natural$ to the PBW basis vectors and consider their images in~$\Delta_q$. We express these images explicitly in the linear basis for $\Delta_q$ mentioned above. Our main results are Theorems~\ref{thm:main1short},~\ref{thm:v3}. These results make heavy use of the Chebyshev polynomials of the second kind~\cite{kls, mason}.

\section{Preliminaries}

We now begin our formal argument. Recall the natural numbers $\mathbb N = \lbrace 0,1,2,\ldots\rbrace$ and integers $\mathbb Z = \lbrace 0,\pm 1, \pm 2,\ldots\rbrace$. Let $\mathbb F$ denote an algebraically closed field with characteristic zero. All the algebras discussed in this paper are over $\mathbb F$; those without the Lie prefix are associative and have a multiplicative identity. Fix a nonzero $q \in \mathbb F$ that is not a root of~1. Recall the notation
\begin{gather}\label{eq:brack}
\lbrack n \rbrack_q = \frac{q^n-q^{-n}}{q-q^{-1}},\qquad n \in \mathbb Z.
\end{gather}
We will be discussing the $q$-Onsager algebra $\mathcal O_q$ and the universal Askey--Wilson algebra $\Delta_q$. We now recall these algebras.

The algebra $\mathcal O_q$ (see \cite[Section~2]{bas1}, \cite[Definition~3.9]{qSerre}) is defined by generators $A$, $B$ and relations
\begin{gather}
A^3B-\lbrack 3\rbrack_q A^2BA+ \lbrack 3\rbrack_q ABA^2 -BA^3 = \big(q^2-q^{-2}\big)^2 (BA-AB), \label{eq:dg1}\\
B^3A-\lbrack 3\rbrack_q B^2AB + \lbrack 3\rbrack_q BAB^2 -AB^3 = \big(q^2-q^{-2}\big)^2 (AB-BA).\label{eq:dg2}
\end{gather}
The relations (\ref{eq:dg1}), (\ref{eq:dg2}) are called the {\it $q$-Dolan/Grady relations}. In~\cite{BK}, Baseilhac and Kolb introduced the automorphisms~$T_0$, $T_1$ of~$\mathcal O_q$. These automorphisms satisfy
\begin{alignat}{3}
& T_0(A)=A, \qquad && T_0(B)= B + \frac{q A^2 B- \big(q+q^{-1}\big)ABA+ q^{-1}BA^2} {\big(q-q^{-1}\big)\big(q^2-q^{-2}\big)},&\label{eq:L0}\\
& T_1(B)=B, \qquad && T_1(A)= A + \frac{q B^2 A- \big(q+q^{-1}\big)BAB+ q^{-1}AB^2} {\big(q-q^{-1}\big)\big(q^2-q^{-2}\big)}.& \label{eq:L1}
\end{alignat}
The inverse automorphisms satisfy
\begin{alignat}{3}
& T^{-1}_0(A)=A, \qquad && T^{-1}_0(B)= B + \frac{q^{-1} A^2 B- \big(q+q^{-1}\big)ABA+ qBA^2} {\big(q-q^{-1}\big)(q^2-q^{-2}\big)},& \label{eq:L0i}\\
& T^{-1}_1(B)=B, \qquad && T^{-1}_1(A)= A + \frac{q^{-1} B^2 A- \big(q+q^{-1}\big)BAB+ qAB^2}{\big(q-q^{-1}\big)\big(q^2-q^{-2}\big)}.& \label{eq:L1i}
\end{alignat}
In \cite{BK}, Baseilhac and Kolb used $T_0$ and $T_1$ to define some elements in $\mathcal O_q$, denoted
\begin{gather}\label{eq:PBW}
\lbrace B_{n \delta + \alpha_0}\rbrace_{n=0}^\infty, \qquad \lbrace B_{n \delta + \alpha_1}\rbrace_{n=0}^\infty, \qquad \lbrace B_{n \delta}\rbrace_{n=1}^\infty.
\end{gather}
The elements
(\ref{eq:PBW}) were shown to be a PBW basis for $\mathcal O_q$, provided that $q$ is transcendental over $\mathbb F$ \cite[Theorem~4.3]{BK}. By definition
\begin{center}
\begin{tabular}[t]{c|cccccc}
 $n$ & $B_{n\delta+\alpha_0}$ &
 $B_{n\delta+\alpha_1}$ \\
\hline
$0$ & $A$ & $B$\\
$1$ &
$T_0 (B)$ &
$T^{-1}_1 (A)$\\
$2$ &
$T_0 T_1 (A)$ &
$T^{-1}_1 T^{-1}_0 (B)$\\
$3$ &
$T_0 T_1 T_0 (B)$ &
$T^{-1}_1 T^{-1}_0 T^{-1}_1 (A)$\\
$4$ &
$T_0 T_1 T_0 T_1 (A)$ &
$T^{-1}_1 T^{-1}_0 T^{-1}_1 T^{-1}_0 (B)$\\
$\vdots $ &
$\vdots $ &
$\vdots $
	 \end{tabular}
\end{center}

\noindent
and for $n\geq 1$,
\begin{gather} \label{eq:Bdelta}
B_{n \delta} = q^{-2} B_{(n-1)\delta+\alpha_1} A - A B_{(n-1)\delta+\alpha_1} + \big(q^{-2}-1\big)\sum_{\ell=0}^{n-2} B_{\ell \delta+\alpha_1} B_{(n-\ell-2) \delta+\alpha_1}.
\end{gather}
By \cite[Proposition~5.12]{BK} the elements $\lbrace B_{n \delta} \rbrace_{n=1}^\infty $ mutually commute. We have $B_{\delta} = q^{-2} BA - AB$. Define $\tilde B_{\delta} = q^{-2} AB-BA$. By \cite[Lemma~3.1]{BK} we have
$T_1 (B_\delta) = \tilde B_{\delta}$ and $T_0 (\tilde B_\delta) = B_{\delta}$. So as noted in \cite[Lemma~3.1]{BK},
\begin{gather}\label{eq:fix}
T_0 T_1 (B_\delta) = B_{\delta}, \qquad T^{-1}_1 T^{-1}_0 (B_\delta) = B_{\delta}.
\end{gather}

Next we recall the universal Askey--Wilson algebra $\Delta_q$ \cite[Definition~1.2]{uaw}. This algebra is defined by generators and relations.
The generators are $A$, $B$, $C$. The relations assert that each of the following is central in $\Delta_q$:
\begin{gather*}
A + \frac{qBC-q^{-1}CB}{q^2-q^{-2}}, \qquad B + \frac{qCA-q^{-1}AC}{q^2-q^{-2}}, \qquad C + \frac{qAB-q^{-1}BA}{q^2-q^{-2}}.
\end{gather*}
For the above three central elements, multiply each by $q+q^{-1}$ to get $\alpha$, $\beta$, $\gamma $. Thus
\begin{gather}
A + \frac{qBC-q^{-1}CB}{q^2-q^{-2}} = \frac{\alpha}{q+q^{-1}},\label{eq:alpha}\\
B + \frac{qCA-q^{-1}AC}{q^2-q^{-2}} =\frac{\beta}{q+q^{-1}},\label{eq:beta}\\
C + \frac{qAB-q^{-1}BA}{q^2-q^{-2}}= \frac{\gamma}{q+q^{-1}}. \label{eq:gamma}
\end{gather}
Each of $\alpha$, $\beta$, $\gamma$ is central in $\Delta_q$. By \cite[Corollary~8.3]{uaw} the center $Z(\Delta_q)$ is generated by $\alpha$, $\beta$, $\gamma$, $\Omega$ where
\begin{gather}
\Omega = qABC+q^2 A^2 + q^{-2} B^2+ q^2 C^2-q A \alpha -q^{-1} B \beta - q C \gamma. \label{eq:Cas}
\end{gather}
The element $\Omega$ is called the Casimir element. By \cite[Theorem~8.2]{uaw} the following is a linear basis for the $\mathbb F$-vector space $Z(\Delta_q)$:
\begin{gather*}
\Omega^\ell \alpha^r \beta^s \gamma^t, \qquad \ell, r,s,t \geq 0.
\end{gather*}
We mention two bases for $\Delta_q$. By \cite[Theorem~4.1]{uaw}, the following is a linear basis for the $\mathbb F$-vector space $\Delta_q$:
\begin{gather}
A^i B^j C^k \alpha^r \beta^s \gamma^t, \qquad i,j,k,r,s,t \geq 0.\label{eq:basispre}
\end{gather}
By \cite[Theorem~7.5]{uaw}, the following is a linear basis for the $\mathbb F$-vector space $\Delta_q$:{\samepage
\begin{gather}
A^i B^j C^k \Omega^\ell \alpha^r \beta^s \gamma^t, \qquad i,j,k,\ell, r,s,t \geq 0, \qquad ijk=0.\label{eq:mainbasis}
\end{gather}
For convenience we will work with the basis (\ref{eq:mainbasis}).}

Shortly we will discuss how $\Delta_q$ is related to $\mathcal O_q$. To aid in this discussion we recall from \cite[Section~2]{uaw} a second presentation of $\Delta_q$. By (\ref{eq:alpha})--(\ref{eq:gamma}) the algebra $\Delta_q$ is generated by~$A$,~$B$,~$\gamma$. Moreover
\begin{gather}
C = \frac{\gamma}{q+q^{-1}} -\frac{qAB-q^{-1}BA}{q^2-q^{-2}},\label{eq:getc}\\
\alpha = \frac{B^2A-\big(q^2+q^{-2}\big)BAB+AB^2+\big(q^2-q^{-2}\big)^2A+\big(q-q^{-1}\big)^2B\gamma}{\big(q-q^{-1}\big)\big(q^2-q^{-2}\big)},\label{eq:getalpha}\\
\beta = \frac{A^2B-\big(q^2+q^{-2}\big)ABA+BA^2+\big(q^2-q^{-2}\big)^2B+\big(q-q^{-1}\big)^2A\gamma}{\big(q-q^{-1}\big)\big(q^2-q^{-2}\big)}.\label{eq:getbeta}
\end{gather}
By \cite[Theorem~2.2]{uaw} the algebra $\Delta_q$ has a presentation by generators $A$, $B$, $\gamma$ and relations
\begin{gather}
A^3B-\lbrack 3\rbrack_q A^2BA+\lbrack 3\rbrack_q ABA^2-BA^3= \big(q^2-q^{-2}\big)^2(BA-AB),\label{eq:rel1}\\
B^3A-\lbrack 3\rbrack_q B^2AB+\lbrack 3\rbrack_q BAB^2-AB^3= \big(q^2-q^{-2}\big)^2(AB-BA),\label{eq:rel2}\\
A^2B^2-B^2A^2+\big(q^2+q^{-2}\big)(BABA-ABAB) = \big(q-q^{-1}\big)^2(BA-AB)\gamma,\label{eq:rel3}\\
 \gamma A = A \gamma, \qquad \gamma B=B \gamma.\label{eq:rel4}
 \end{gather}
The relations (\ref{eq:rel1}), (\ref{eq:rel2}) are the $q$-Dolan/Grady relations. Consequently there exists an algebra homomorphism $\natural\colon \mathcal O_q \to \Delta_q$ that sends $A\mapsto A$ and $B\mapsto B$. This homomorphism is not injective by \cite[Theorem~10.9]{uaw}.

In order to clarify the nature of $T_0$, $T_1$, $\natural$ we now introduce some automorphisms $t_0$, $t_1$ of~$\Delta_q$ such that $t_0 \natural = \natural T_0$ and $t_1 \natural = \natural T_1$. To this end, we recall from \cite[Section~3]{uaw} how the modular group ${\rm PSL}_2(\mathbb Z)$ acts on $\Delta_q$ as a group of automorphisms. By \cite{RCA} the group ${\rm PSL}_2(\mathbb Z)$ has a~presentation by generators $\rho$, $\sigma$ and relations $\rho^3=1$, $\sigma^2=1$. Earlier in this section we gave two presentations of $\Delta_q$. Using these presentations we find that ${\rm PSL}_2(\mathbb Z)$ acts on $\Delta_q$ as a~group of automorphisms in the following way:
\begin{gather}\label{eq:table}
\begin{tabular}{c| ccc | c c c}
$u$ & $A$ & $B$ & $C$
& $\alpha $ & $\beta $ & $\gamma$\\
\hline
$\rho(u)$ & $B$ & $C$ & $A$
& $\beta $ & $\gamma $ & $\alpha$\\
$\sigma(u) $ & $B$ & $A$ & $C+\frac{AB-BA}{q-q^{-1}}$
& $\beta $ & $\alpha $ & $\gamma$
\end{tabular}
\end{gather}

This action is faithful by \cite[Theorem 3.13]{uaw}. From the table (\ref{eq:table}) we see that the ${\rm PSL}_2(\mathbb Z)$-generators $\rho$, $\sigma$ each permute $ \alpha$, $\beta$, $\gamma$. This gives a group homomorphism from ${\rm PSL}_2(\mathbb Z)$ onto the symmetric group $S_3$. Let $\mathbb P$ denote the kernel of the homomorphism. Thus $\mathbb P$ is a normal subgroup of ${\rm PSL}_2(\mathbb Z)$, and the quotient group ${\rm PSL}_2(\mathbb Z)/\mathbb P$ is isomorphic to $S_3$. The cosets of $\mathbb P$ in ${\rm PSL}_2(\mathbb Z)$ are
\begin{gather*}
\mathbb P, \qquad \rho \mathbb P, \qquad \rho^2 \mathbb P, \qquad \sigma \mathbb P, \qquad \rho \sigma \mathbb P, \qquad \rho^2 \sigma \mathbb P.
\end{gather*}
We remark that in the literature the groups ${\rm PSL}_2(\mathbb Z)$ and $\mathbb P$ are often denoted by $\Gamma$ and $\Gamma(2)$, respectively; see for example \cite{RCA,RCA2}. Define
\begin{gather}\label{eq:t0t1}
t_0 = \big(\rho^2 \sigma\big)^2 = (\sigma \rho)^{-2},\qquad t_1 = \big(\sigma \rho^2\big)^2 = (\rho \sigma)^{-2}.
\end{gather}
Using (\ref{eq:table}), (\ref{eq:t0t1}) we obtain $t_0, t_1 \in \mathbb P$. By \cite[Proposition~4]{RCA2} the group $\mathbb P$ is freely generated by $t^{\pm 1}_0$, $t^{\pm 1}_1$. Using (\ref{eq:getc}), (\ref{eq:table}), (\ref{eq:t0t1}) we obtain
\begin{alignat}{3}
& t_0(A)=A, \qquad && t_0(B)= B + \frac{q A^2 B- \big(q+q^{-1}\big)ABA+ q^{-1}BA^2} {\big(q-q^{-1}\big)\big(q^2-q^{-2}\big)},& \label{eq:t0}\\
& t_1(B)=B, \qquad && t_1(A)= A + \frac{q B^2 A- \big(q+q^{-1}\big)BAB+ q^{-1}AB^2}{\big(q-q^{-1}\big)\big(q^2-q^{-2}\big)}& \label{eq:t1}
\end{alignat}
and
\begin{alignat}{3}
& t^{-1}_0(A)=A, \qquad && t^{-1}_0(B)= B + \frac{q^{-1} A^2 B- \big(q+q^{-1}\big)ABA+ qBA^2}{\big(q-q^{-1}\big)\big(q^2-q^{-2}\big)},& \label{eq:t0i}\\
& t^{-1}_1(B)=B, \qquad && t^{-1}_1(A)= A + \frac{q^{-1} B^2 A- \big(q+q^{-1}\big)BAB+ qAB^2}{\big(q-q^{-1}\big)\big(q^2-q^{-2}\big)}.& \label{eq:t1i}
\end{alignat}
The actions (\ref{eq:t0})--(\ref{eq:t1i}) match (\ref{eq:L0})--(\ref{eq:L1i}). Consequently the following diagrams commute:
 \begin{gather}
 { \begin{CD}
 {\mathcal O}_q @>\natural>>
 \Delta_q \\
 @V T^{\pm 1}_1 VV @VV t^{\pm 1}_1 V \\
 {\mathcal O}_q
 @>\natural>>
 \Delta_q,
 \end{CD}
 } \qquad \qquad\qquad {
 \begin{CD}
 {\mathcal O}_q @>\natural>>
 \Delta_q \\
 @V T^{\pm 1}_0 VV @VV t^{\pm 1}_0 V \\
 {\mathcal O}_q
 @>\natural>>
 \Delta_q.
 \end{CD}}\label{eq:cd}
\end{gather}
Let $\operatorname{Aut}(\mathcal O_q)$ denote the automorphism group of~$\mathcal O_q$. Let $G$ denote the subgroup of $\operatorname{Aut}(\mathcal O_q)$ generated by $T^{\pm 1}_0$, $T^{\pm 1}_1$. Since~$\mathbb P$ is freely generated by $t^{\pm 1}_0$, $t^{\pm 1}_1$ there exists a group homomorphism $\varepsilon \colon \mathbb P\to G$ that sends $t^{\pm 1}_0 \mapsto T^{\pm 1}_0$ and $t^{\pm 1}_1 \mapsto T^{\pm 1}_1$. Using the commuting diagrams~(\ref{eq:cd}) one finds that for $\pi \in \mathbb P$ the following diagram commutes:
 \begin{gather}
 { \begin{CD}
 {\mathcal O}_q @>\natural>>
 \Delta_q \\
 @V \varepsilon (\pi) VV @VV \pi V \\
 {\mathcal O}_q
 @>\natural>>
 \Delta_q.
 \end{CD} }\label{eq:epsIso}
\end{gather}
We now prove that $\varepsilon$ is an isomorphism. By construction $\varepsilon$ is surjective. We show that $\varepsilon$ is injective. Given an element $r$ in the kernel of $\varepsilon$, we show that $r$ is the identity in~$\mathbb P$. To this end, we show that $r$ fixes the generators $A$, $B$, $\gamma$ of $\Delta_q$. The map $\varepsilon(r)$ is the identity in $G$, so $\varepsilon(r)$ fixes the elements $A$, $B$ of~$\mathcal O_q$. By the commuting diagram (\ref{eq:epsIso}) the map $r$ fixes the elements $A$, $B$ of $\Delta_q$. Also $r$ fixes $\gamma$ since $r \in \mathbb P$ and everything in $\mathbb P$ fixes~$\gamma$. We have shown that~$r$ fixes the generators $A$, $B$, $\gamma$ of~$\Delta_q$ so $r$ is the identity in $\mathbb P$. Consequently $\varepsilon $ is injective and hence an isomorphism.

It is mentioned in \cite[Section~2.3]{BK} that one expects $G$ to be freely generated by $T^{\pm 1}_0$, $T^{\pm 1}_1$. This is now easily proven as follows. The group $\mathbb P$ is freely generated by~$t^{\pm 1}_0$,~$t^{\pm 1}_1$. Applying the isomorphism $\varepsilon\colon \mathbb P \to G$ we find that $G$ is freely generated by $T^{\pm 1}_0$, $T^{\pm 1}_1$.

Next we consider how the map $\natural\colon \mathcal O_q \to \Delta_q$ acts on the elements~(\ref{eq:PBW}). For these elements we retain the same notation for their images under $\natural$. Our goal is to obtain these images in closed form, in terms of the basis~(\ref{eq:mainbasis}). In order to obtain these images, it is convenient to bring in the Chebyshev polynomials of the second kind. These polynomials are reviewed in the next section.

\section{The Chebyshev polynomials}
In this section we review the Chebyshev polynomials of the second kind; see \cite{kls, mason} for further details. Let $x$ denote an indeterminate. Let $\mathbb F \lbrack x \rbrack$ denote the $\mathbb F$-algebra consisting of the polynomials in $x$ that have all coefficients in $\mathbb F$.

\begin{Definition}[{see \cite[p.~4]{mason}}] \label{def:Tch} For $n \in \mathbb N$ define $U_n \in \mathbb F \lbrack x \rbrack$ by
\begin{gather*}
U_0 = 1, \qquad U_1 = x,\qquad x U_n = U_{n+1} + U_{n-1}, \qquad n\geq 1.
\end{gather*}
The polynomial $U_n$ is monic and degree $n$. We call $U_n$ the {\it $n$th Chebyshev polynomial of the second kind}. For notational convenience define $U_{n}=0$ for all integers $n<0$.
\end{Definition}

\begin{Note}The above polynomials $U_n$ are normalized to be monic. This normalization differs from the one in \cite[Section 9.8.2]{kls}. To go from our normalization to the one in \cite[Sec\-tion~9.8.2]{kls}, replace~$x$ by~$2x$.
\end{Note}

In the table below we display $U_n$ for $0 \leq n \leq 9$.
\begin{center}
 \begin{tabular}[t]{c|c}
 $n$ & $U_n$ \\
\hline
$0$ & $1$\\
$1$ & $x$\\
$2$ & $x^2-1$\\
$3$ & $x^3-2x$\\
$4$ & $x^4-3x^2+1$\\
$5$ & $x^5-4x^3+3x$\\
$6$ & $x^6-5x^4+6x^2-1$\\
$7$ & $x^7-6x^5+10x^3-4x$\\
$8$ & $x^8-7x^6+15x^4-10x^2+1$\\
$9$ & $x^9-8x^7+21x^5-20x^3+5x$
\end{tabular}
\end{center}

By \cite[pp.~332--333]{mason},
\begin{gather*}
 U_n(x)= \sum_{i=0}^{\lfloor n/2 \rfloor} (-1)^i \binom{n-i}{i} x^{n-2i}, \qquad n \in \mathbb N.
\end{gather*}
Next we express the polynomials $U_n$ in a more closed form. Let $z$ denote an indeterminate. Let $\mathbb F\lbrack z, z^{-1}\rbrack$ denote the $\mathbb F$-algebra consisting of the Laurent polynomials in $z$ that have all coefficients in $\mathbb F$. This algebra has an automorphism that sends $z\mapsto z^{-1}$. An element of $\mathbb F\lbrack z, z^{-1}\rbrack$ that is fixed by the automorphism is called {\it symmetric}. The symmetric elements form a subalgebra of $\mathbb F\lbrack z, z^{-1}\rbrack$ called its {\it symmetric part}. There exists an injective algebra homomorphism $\iota\colon \mathbb F\lbrack x \rbrack \to \mathbb F\lbrack z, z^{-1}\rbrack$ that sends $x\mapsto z+z^{-1}$. The image of $\mathbb F\lbrack x \rbrack$ under $\iota$ is the symmetric part of $\mathbb F\lbrack z, z^{-1}\rbrack$. Via $\iota$ we identify $\mathbb F\lbrack x \rbrack$ with the symmetric part of $\mathbb F\lbrack z, z^{-1}\rbrack$. So for $n \in \mathbb N$ we view
\begin{gather*}
 \frac{z^{n+1}-z^{-n-1}}{z-z^{-1}} = z^n + z^{n-2} + \cdots + z^{2-n}+ z^{-n}
\end{gather*}
as an element of $\mathbb F\lbrack x \rbrack$.

\begin{Lemma}[{see \cite[p.~326]{mason}}] \label{lem:zform} For $n \in \mathbb N$ we have
\begin{gather*}
U_n(x) = \frac{z^{n+1}-z^{-n-1}}{z-z^{-1}},
\end{gather*}
where we recall $x=z+z^{-1}$.
\end{Lemma}

In this paper, on several occasions we will consider generating functions in an indeterminate $t$. These generating functions involve a formal power series; issues of convergence are not considered. The following generating function will be useful.
\begin{Lemma}[{see \cite[p.~227]{kls}}] \label{lem:gf} For an indeterminate $t$,
\begin{gather}\label{eq:gen}
\sum_{n\in \mathbb N} t^n U_n(x) = \frac{1}{1-t x + t^2}.
\end{gather}
\end{Lemma}
\begin{proof} Using Definition~\ref{def:Tch} one finds that the product $\big(\sum\limits_{n \in \mathbb N} t^n U_n(x) \big) \big(1-tx+t^2\big)$ is equal to~1. Alternatively, use Lemma~\ref{lem:zform}.
\end{proof}

The following variations on Lemma~\ref{lem:gf} will be used repeatedly.
\begin{Lemma} \label{lem:var} For an indeterminate $t$,
\begin{gather*}
\sum_{n \in \mathbb N} (-1)^n q^n t^n U_{n-1}(x) = \frac{-1}{qt+q^{-1} t^{-1} + x}, \\
\sum_{n \in \mathbb N} (-1)^n q^{-n} t^n U_{n-1}(x) = \frac{-1}{q^{-1}t+q t^{-1} + x}.
\end{gather*}
\end{Lemma}
\begin{Lemma} \label{lem:Tbrack} For an indeterminate $t$,
\begin{gather*}
\sum_{n \in \mathbb N} (-1)^n t^n \lbrack n \rbrack_q U_{n-1}(x) = \frac{t-t^{-1}}{ \big(q t+ q^{-1} t^{-1} + x\big) \big(q^{-1} t+ q t^{-1} + x\big)}.
\end{gather*}
\end{Lemma}
\begin{proof} Observe that
\begin{align*}
\sum_{n \in \mathbb N} (-1)^n t^n \lbrack n \rbrack_q U_{n-1}(x)&=\sum_{n \in \mathbb N} (-1)^n t^n \frac{q^n-q^{-n}}{q-q^{-1}} U_{n-1}(x)\\
&=\sum_{n \in \mathbb N} \frac{(-1)^n t^n q^n U_{n-1}(x)}{q-q^{-1}}-\sum_{n \in \mathbb N} \frac{(-1)^n t^n q^{-n} U_{n-1}(x)}{q-q^{-1}}\\
&=\frac{1}{q-q^{-1}} \frac{-1}{qt+q^{-1}t^{-1}+x}-\frac{1}{q-q^{-1}} \frac{-1}{q^{-1}t+qt^{-1}+x}\\
& = \frac{t-t^{-1}}{\big(qt+q^{-1}t^{-1}+x\big)\big(q^{-1}t+qt^{-1}+x\big)}.\tag*{\qed}
\end{align*}\renewcommand{\qed}{}
\end{proof}

\section{Some identities}

In this section we give some identities for later use.

\begin{Lemma}\label{lem:brRec}
For $r \in \mathbb Z$,
\begin{gather*}
\lbrack r-1\rbrack_q-\big(q+q^{-1}\big) \lbrack r \rbrack_q+ \lbrack r+1 \rbrack_q = 0.
\end{gather*}
\end{Lemma}
\begin{proof} Use (\ref{eq:brack}).
\end{proof}

\begin{Lemma} \label{lem:ident3} For $r,s\in \mathbb Z$ we have
\begin{gather*}
\lbrack r-1\rbrack_q \lbrack s-1\rbrack_q \lbrack r-s\rbrack_q+ \lbrack r\rbrack_q \lbrack s\rbrack_q \lbrack r-s\rbrack_q = \lbrack r-1\rbrack_q \lbrack s\rbrack_q \lbrack r-s+1\rbrack_q
+ \lbrack r\rbrack_q \lbrack s-1\rbrack_q \lbrack r-s-1\rbrack_q.
\end{gather*}
\end{Lemma}
\begin{proof} Expand each side using~(\ref{eq:brack}).
\end{proof}

\begin{Lemma} For an indeterminate $t$,
\begin{alignat*}{3}
&\sum_{\ell \in \mathbb N} t^{2 \ell} = \frac{-t^{-1}}{t-t^{-1}}, \qquad && \sum_{\ell \in \mathbb N}\ell t^{2 \ell} = \frac{1}{\big(t-t^{-1}\big)^2},& \\
&\sum_{\ell \in \mathbb N} \ell^2 t^{2 \ell} = - \frac{t+t^{-1}}{\big(t-t^{-1}\big)^3},\qquad && \sum_{\ell \in \mathbb N}\binom{\ell+1}{2} t^{2 \ell+1} = \frac{-1}{\big(t-t^{-1}\big)^3}.&
\end{alignat*}
\end{Lemma}
\begin{proof} These are readily checked.
\end{proof}

\section{The main results}

In this section we express the images (\ref{eq:PBW}) in the basis (\ref{eq:mainbasis}). In what follows, the notation $\lbrack u,v\rbrack $ means $uv-vu$. We will use a recursion found in \cite{BK}; we give a short proof for the sake of completeness.

\begin{Lemma}[{see \cite[Section~3.1]{BK}}]\label{lem:prelim} In the algebra ${\mathcal O}_q$,
\begin{gather}
B_{\alpha_0}=A, \qquad B_{\delta+\alpha_0} = B +\frac{q \lbrack B_{\delta}, A\rbrack}{\big(q-q^{-1}\big)\big(q^2-q^{-2}\big)},\label{eq:line1}\\
B_{n \delta+\alpha_0} = B_{(n-2)\delta+\alpha_0}+\frac{q \lbrack B_{\delta}, B_{(n-1)\delta+\alpha_0}\rbrack}{\big(q-q^{-1}\big)\big(q^2-q^{-2}\big)}, \qquad n\geq 2\label{eq:line2}
\end{gather}
and also
\begin{gather}
B_{\alpha_1}=B, \qquad B_{\delta+\alpha_1} = A -\frac{q \lbrack B_{\delta}, B\rbrack}{\big(q-q^{-1}\big)\big(q^2-q^{-2}\big)},\label{eq:line3}\\
B_{n \delta+\alpha_1} = B_{(n-2)\delta+\alpha_1}-\frac{q \lbrack B_{\delta}, B_{(n-1)\delta+\alpha_1}\rbrack}{\big(q-q^{-1}\big)\big(q^2-q^{-2}\big)}, \qquad n\geq 2.\label{eq:line4}
\end{gather}
\end{Lemma}
\begin{proof} We show that
\begin{gather}\label{eq:L0B}
T_0(B) = B + \frac{q\lbrack B_\delta, A\rbrack}{\big(q-q^{-1}\big)\big(q^2-q^{-2}\big)},\\
\label{eq:L1A}T^{-1}_1(A)= A -\frac{q\lbrack B_\delta, B\rbrack}{\big(q-q^{-1}\big)\big(q^2-q^{-2}\big)}.
\end{gather}
To verify
(\ref{eq:L0B}) (resp.~(\ref{eq:L1A})) eliminate $B_\delta$ using $B_\delta = q^{-2} BA-AB$ and compare the result with~(\ref{eq:L0}) (resp.~(\ref{eq:L1i})). Lines~(\ref{eq:line1}),~(\ref{eq:line3}) follow from~(\ref{eq:L0B}),~(\ref{eq:L1A}) and the construction. Now consider~(\ref{eq:line2}),~(\ref{eq:line4}). First assume that $n=2r+1$ is odd. To verify~(\ref{eq:line2}), apply $(T_0 T_1)^r$ to each side of~(\ref{eq:L0B}), and use~(\ref{eq:fix}) along with $T_1(B)=B$. To verify~(\ref{eq:line4}), apply $(T_0 T_1)^{-r}$ to each side of (\ref{eq:L1A}), and use~(\ref{eq:fix}) along with $T_0(A)=A$. Next assume that $n=2r$ is even. To verify~(\ref{eq:line2}), apply $(T_0 T_1)^r$ to each side of~(\ref{eq:L1A}), and use~(\ref{eq:fix}) along with $T_0(A)=A$, $T_1(B)=B$. To verify~(\ref{eq:line4}), apply $(T_0 T_1)^{-r}$ to each side of~(\ref{eq:L0B}), and use~(\ref{eq:fix}) along with $T_0(A)=A$, $T_1(B)=B$.
\end{proof}

\begin{Lemma}\label{lem:BdeltaC}In the algebra $\Delta_q$,
\begin{gather}
B_{\delta} = q^{-1}\big(q^2-q^{-2}\big)C-q^{-1}\big(q-q^{-1}\big)\gamma.\label{eq:line5}
\end{gather}
\end{Lemma}
\begin{proof} Simplify (\ref{eq:gamma}) using $qAB-q^{-1}BA= -q B_\delta$.
\end{proof}

\begin{Lemma}\label{lem:prelim2}In the algebra ${\Delta_q}$,
\begin{gather}
B_{\alpha_0}=A, \qquad B_{\delta+\alpha_0} = B +\frac{\lbrack C, A\rbrack}{q-q^{-1}},\label{eq:pre1}\\
B_{n \delta+\alpha_0} = B_{(n-2)\delta+\alpha_0}+\frac{ \lbrack C, B_{(n-1)\delta+\alpha_0}\rbrack}{q-q^{-1}}, \qquad n\geq 2\label{eq:pre2}
\end{gather}
and also
\begin{gather}
B_{\alpha_1}=B, \qquad B_{\delta+\alpha_1} = A -\frac{ \lbrack C, B\rbrack}{q-q^{-1}},\label{eq:pre3}\\
B_{n \delta+\alpha_1} = B_{(n-2)\delta+\alpha_1}-\frac{\lbrack C, B_{(n-1)\delta+\alpha_1}\rbrack}{q-q^{-1}},\qquad n\geq 2.\label{eq:pre4}
\end{gather}
\end{Lemma}
\begin{proof} Evaluate (\ref{eq:line1})--(\ref{eq:line4}) using (\ref{eq:line5}) and the fact that $\gamma$ is central in $\Delta_q$.
\end{proof}

\begin{Lemma} \label{lem:CACB}In the algebra $\Delta_q$,
\begin{gather}\label{eq:comCA}
\frac{\lbrack C, A \rbrack}{q-q^{-1}} = -q^{-1} AC -q^{-1} \big(q+q^{-1}\big) B +q^{-1} \beta,\\
\label{eq:comCB}\frac{\lbrack C, B \rbrack}{q-q^{-1}} = q BC +q \big(q+q^{-1}\big) A -q \alpha.
\end{gather}
\end{Lemma}
\begin{proof} These equations are a reformulation of~(\ref{eq:alpha}),~(\ref{eq:beta}).
\end{proof}

The following is our first main result.

\begin{Theorem}\label{thm:main1short}For $n\geq 0$ the following hold in~$\Delta_q$:
\begin{gather*}
B_{n\delta+\alpha_0} = (-1)^n q^{-n} A U_n(C)+(-1)^n q^{-n-1} B U_{n-1}(C)+(-1)^n \alpha \sum_{j \in \mathbb N} q^{2j-n+1} U_{n-2j-2}(C)\\
\hphantom{B_{n\delta+\alpha_0} =}{} + (-1)^{n-1} \beta \sum_{j\in \mathbb N} q^{2j-n} U_{n-2j-1}(C),\\
B_{n\delta+\alpha_1} = (-1)^n q^{n} B U_n(C)+ (-1)^n q^{n+1} A U_{n-1}(C)+(-1)^n \beta \sum_{j \in \mathbb N}q^{n-2j-1} U_{n-2j-2}(C)\\
\hphantom{B_{n\delta+\alpha_1} =}{} + (-1)^{n-1}\alpha\sum_{j\in \mathbb N} q^{n-2j} U_{n-2j-1}(C).
\end{gather*}
\end{Theorem}
\begin{proof} By a routine induction on $n$, using Lemmas~\ref{lem:prelim2},~\ref{lem:CACB}.
\end{proof}

The following is our second main result.

\begin{Theorem}\label{thm:v3}In the algebra $\Delta_q$, for $n\geq 1$ the element $B_{n\delta}$ is equal to $(-1)^n\big(1-q^{-2}\big)$ times a weighted sum with the following terms and coefficients:
\begin{gather*}
 \begin{tabular}[t]{c|c}
 {\rm term} & {\rm coefficient} \\
\hline
$\Omega$ & $ \sum\limits_{\ell \in \mathbb N}\lbrack n-2\ell-1 \rbrack_q U_{n-2\ell-2}(C)$\\
$\alpha \beta $ &$ \sum\limits_{\ell \in \mathbb N}\ell^2 \lbrack n-2\ell\rbrack_qU_{n-2\ell-1}(C)$\\
$\alpha^2 + \beta^2$ &$-\sum\limits_{\ell \in \mathbb N}\binom{\ell+1}{2} \lbrack n-2\ell-1\rbrack_q U_{n-2\ell-2}(C)$\\
$\gamma$ &$\lbrack n \rbrack_q U_{n-1}(C)+2\sum\limits_{\ell\in \mathbb N} \lbrack n-2\ell-2\rbrack_q U_{n-2\ell-3}(C)$\\
$1$ &$ -\lbrack n+1 \rbrack_q U_n(C)-\lbrack 3 \rbrack_q \lbrack n-1 \rbrack_q U_{n-2}(C)- \lbrack 2 \rbrack^2_q\sum\limits_{\ell \in \mathbb N } \lbrack n-2\ell-3 \rbrack_q U_{n-2\ell-4}(C)$
\end{tabular}
\end{gather*}
\end{Theorem}

\begin{proof}We have some preliminary comments. Using (\ref{eq:beta}), (\ref{eq:gamma}),
\begin{gather*}
BA = q^2 AB+q\big(q^2-q^{-2}\big)C - q\big(q-q^{-1}\big) \gamma,\\
CA = q^{-2} AC-q^{-1}\big(q^2-q^{-2}\big)B + q^{-1}\big(q-q^{-1}\big) \beta,\\
CA^2 = q^{-4} A^2C - q^{-1} \big(q^4-q^{-4}\big) AB + q^{-2} \big(q^2-q^{-2}\big)A\beta\\
\hphantom{CA^2 =}{} -\big(q^2-q^{-2}\big)^2 C + \big(q-q^{-1}\big)\big(q^2-q^{-2}\big)\gamma.
\end{gather*}
By \cite[Lemma~6.1]{uaw},
\begin{gather*}
BAC = q \Omega - q^3 A^2 - q^{-1} B^2 - q^{-1} C^2 + q^2 A \alpha + B\beta + C \gamma,\\
CAB = q^{-1} \Omega - q^{-3} A^2 - q B^2 - q C^2 + q^{-2} A \alpha+ B\beta + C \gamma.
\end{gather*}
We are done with the preliminary comments. We now define some generating functions in an indeterminate $t$:
\begin{gather}\label{eq:genf}
\Phi(t) = \sum_{n=0}^\infty t^n B_{n \delta + \alpha_1},\qquad\Psi(t) = \sum_{n=1}^\infty t^n B_{n \delta}.
\end{gather}
By (\ref{eq:Bdelta}),
\begin{gather}
\Psi(t) = q^{-2} t \Phi(t) A - t A \Phi(t) + \big(q^{-2}-1\big) t^2 (\Phi(t))^2. \label{eq:BdeltaGF}
\end{gather}
By (\ref{eq:pre3}), (\ref{eq:pre4}),
\begin{gather}
\frac{\lbrack C, \Phi(t) \rbrack}{q-q^{-1}} = A + t^{-1} B + \big(t-t^{-1}\big) \Phi(t).\label{eq:comGF}
\end{gather}
We next consider what the second equation in Theorem~\ref{thm:main1short} implies about $\Phi(t)$. Using Lem\-ma~\ref{lem:gf},
\begin{gather*}
\sum_{n \in \mathbb N}(-1)^n q^n t^n U_n(x)=\frac{q^{-1} t^{-1}} {qt+q^{-1} t^{-1} + x }.
\end{gather*}
Using Lemma \ref{lem:var},
\begin{gather*}
\sum_{n \in \mathbb N}(-1)^n q^{n+1} t^n U_{n-1}(x)=\frac{-q}{qt+q^{-1} t^{-1} + x}.
\end{gather*}
We have
\begin{gather*}
\sum_{n \in \mathbb N}(-1)^n t^n \sum_{j \in \mathbb N}q^{n-2j-1} U_{n-2j-2}(x)\\
\qquad{} = \sum_{n \in \mathbb N}\sum_{j\in \mathbb N}(-1)^n t^n q^{n-2j-1} U_{n-2j-2}(x)\\
 \qquad {} =-\sum_{n \in \mathbb N} \sum_{j\in \mathbb N}(-1)^{n-2j-1} t^{n-2j-1}q^{n-2j-1} U_{n-2j-2}(x) t^{2j+1}\\
\qquad {} =- \sum_{N \in \mathbb N} \sum_{j \in \mathbb N} (-1)^{N} t^{N}q^N U_{N-1}(x) t^{2j+1}\qquad {\mbox{\rm (change var. $N=n-2j-1$)}}\\
 \qquad {} = - \Biggl( \sum_{N \in \mathbb N}(-1)^{N} t^{N}q^N U_{N-1}(x)\Biggr)\Biggl(\sum_{j \in \mathbb N} t^{2j+1} \Biggr)\\
 \qquad {} = \frac{-1}{qt+q^{-1}t^{-1} + x}\frac{1}{t-t^{-1}}\\
\qquad {} = \frac{-1}{\big(t-t^{-1}\big)\big(qt+q^{-1} t^{-1} + x\big)}.
\end{gather*}
Similarly,
\begin{gather*}
\sum_{n \in \mathbb N}(-1)^{n-1} t^n \sum_{j \in \mathbb N} q^{n-2j} U_{n-2j-1}(x)=\frac{-t^{-1}}{\big(t-t^{-1}\big)\big(qt+q^{-1} t^{-1} + x\big)}.
\end{gather*}
By these comments and the second equation in Theorem~\ref{thm:main1short},
\begin{gather}
\Phi(t) \big(qt+q^{-1} t^{-1} + C\bigr) = q^{-1} t^{-1} B - q A - \frac{\beta}{t-t^{-1}}- \frac{t^{-1} \alpha}{t-t^{-1}}.\label{eq:PhiC}
\end{gather}
By (\ref{eq:comGF}) and (\ref{eq:PhiC}),
\begin{gather}
\big(q^{-1} t+q t^{-1} + C \big)\Phi(t) =q t^{-1} B - q^{-1} A - \frac{\beta}{t-t^{-1}}- \frac{t^{-1} \alpha}{t-t^{-1}}.\label{eq:CPhi}
\end{gather}
In (\ref{eq:BdeltaGF}), we multiply each side on the left by $q^{-1} t + q t^{-1} + C$ and on the right by $q t + q^{-1} t^{-1} + C$. We evaluate the result using~(\ref{eq:PhiC}),~(\ref{eq:CPhi}) to obtain
\begin{gather*}
\bigl(q^{-1} t + q t^{-1} + C\bigr)\Psi(t)\bigl(q t + q^{-1} t^{-1} + C\bigr)\\
\qquad{} = q^{-2} t \left(q t^{-1} B - q^{-1} A - \frac{\beta}{t-t^{-1}}- \frac{t^{-1} \alpha}{t-t^{-1}} \right)A\bigl(q t + q^{-1} t^{-1} + C\bigr)\\
\qquad\quad{} -t \bigl(q^{-1} t + q t^{-1} + C\bigr)A\left(q^{-1} t^{-1} B - q A - \frac{\beta}{t-t^{-1}}- \frac{t^{-1} \alpha}{t-t^{-1}}\right)\\
\qquad\quad{} + \big(q^{-2}-1\big)t^2 \left(q t^{-1} B - q^{-1} A - \frac{\beta}{t-t^{-1}}- \frac{t^{-1} \alpha}{t-t^{-1}}\right)\\
\qquad\quad{}\times \left(q^{-1} t^{-1} B - q A - \frac{\beta}{t-t^{-1}} - \frac{t^{-1} \alpha}{t-t^{-1}}\right).
\end{gather*}
Evaluating the above equation using the preliminary comments, we find that
\begin{gather}
\bigl(q^{-1} t + q t^{-1} + C\bigr)\Psi(t) \bigl(q t + q^{-1} t^{-1} + C\bigr)\label{eq:3}
\end{gather}
is equal to $1-q^{-2}$ times
\begin{gather*}
 \Omega - \frac{\big(t+t^{-1}\big) \alpha \beta }{\big(t-t^{-1}\big)^2}-\frac{\alpha^2+ \beta^2}{\big(t-t^{-1}\big)^2}- \big(t+t^{-1}\big)\gamma+\big(q+q^{-1}\big)\big(t+t^{-1}\big) C + C^2.
\end{gather*}
Consequently $\Psi(t)$ is equal to $1-q^{-2}$ times
\begin{gather*}
F_1(t,C) \Omega+ F_2(t,C) \alpha \beta +F_3(t,C) \big(\alpha^2+\beta^2\big)+F_4(t,C) \gamma+F_5(t,C),
\end{gather*}
where
\begin{gather*}
F_1(t,x) = \frac{1}{\big(qt+q^{-1}t^{-1} + x\big) \big(q^{-1}t+qt^{-1} + x\big)},\\
F_2(t,x) = -\frac{t+t^{-1}}{\big(t-t^{-1}\big)^2\big(qt+q^{-1}t^{-1} + x\big)\big(q^{-1}t+qt^{-1} + x\big)},\\
F_3(t,x) = \frac{-1}{\big(t-t^{-1}\big)^2\big(qt+q^{-1}t^{-1} + x\big)\big(q^{-1}t+qt^{-1} + x\big)},\\
F_4(t,x) = -\frac{t+t^{-1}}{\big(qt+q^{-1}t^{-1} + x\big)\big(q^{-1}t+qt^{-1} + x\big)},\\
F_5(t,x) = \frac{\big(q+q^{-1}\big)\big(t+t^{-1}\big)x+x^2}{\big(qt+q^{-1}t^{-1} + x\big)\big(q^{-1}t+qt^{-1} + x\big)}.
\end{gather*}
We now compare the $\lbrace F_i\rbrace_{i=1}^5$ with the coefficients shown in the table of the theorem statement.

{\allowdisplaybreaks Concerning $F_1$,
\begin{gather*}
 \sum_{n=1}^\infty (-1)^n t^n \sum_{\ell \in \mathbb N} \lbrack n-2\ell-1\rbrack_q U_{n-2\ell-2}(x)\\
\qquad{} = \sum_{n=1}^\infty \sum_{\ell \in \mathbb N}(-1)^{n} t^{n}\lbrack n-2\ell-1\rbrack_q U_{n-2\ell-2}(x)\\
\qquad{} = - \sum_{n=1}^\infty \sum_{\ell \in \mathbb N}(-1)^{n-2\ell-1} t^{n-2\ell-1}\lbrack n-2\ell-1\rbrack_q U_{n-2\ell-2}(x)t^{2\ell+1}\\
\qquad{} = - \sum_{N \in \mathbb N} \sum_{\ell \in \mathbb N}(-1)^{N} t^{N}\lbrack N \rbrack_q U_{N-1}(x) t^{2\ell+1}\qquad {\mbox{\rm (change var. $N=n-2\ell-1$)}}\\
\qquad{} = - \Biggl( \sum_{N \in \mathbb N}(-1)^{N} t^{N}\lbrack N \rbrack_q U_{N-1}(x)\Biggr)\Biggl(\sum_{\ell \in \mathbb N} t^{2\ell+1} \Biggr)\\
\qquad{} = \frac{t-t^{-1}}{\big(qt+q^{-1}t^{-1} + x\big)\big(q^{-1}t+qt^{-1} + x\big)}\frac{1}{t-t^{-1}}\\
\qquad{} = \frac{1}{\big(qt+q^{-1}t^{-1} + x\big)\big(q^{-1}t+qt^{-1} + x\big)}\\
\qquad{} = F_1(t,x).
\end{gather*}
Concerning $F_2$,
\begin{gather*}
 \sum_{n=1}^\infty (-1)^n t^n \sum_{\ell \in \mathbb N} \ell^2\lbrack n-2\ell\rbrack_q U_{n-2\ell-1}(x)\\
\qquad{} = \sum_{n=1}^\infty \sum_{\ell \in \mathbb N}(-1)^{n} t^{n}\lbrack n-2\ell\rbrack_q U_{n-2\ell-1}(x) \ell^2\\
\qquad{} = \sum_{n=1}^\infty \sum_{\ell \in \mathbb N}(-1)^{n-2\ell} t^{n-2\ell}\lbrack n-2\ell\rbrack_q U_{n-2\ell-1}(x) \ell^2 t^{2\ell}\\
\qquad{} = \sum_{N \in \mathbb N} \sum_{\ell \in \mathbb N}(-1)^{N} t^{N}\lbrack N \rbrack_q U_{N-1}(x) \ell^2 t^{2\ell} \qquad {\mbox{\rm (change var. $N=n-2\ell$)}}\\
\qquad{} = \Biggl( \sum_{N \in \mathbb N}(-1)^{N} t^{N}\lbrack N \rbrack_q U_{N-1}(x)\Biggr)\Biggl(\sum_{\ell \in \mathbb N}\ell^2 t^{2\ell} \Biggr)\\
\qquad{} = -\frac{t-t^{-1}}{\big(qt+q^{-1}t^{-1} + x\big)\big(q^{-1}t+qt^{-1} + x\big)}\frac{t+t^{-1}}{\big(t-t^{-1}\big)^3}\\
\qquad{} = -\frac{t+t^{-1}}{\big(t-t^{-1}\big)^2 \big(qt+q^{-1}t^{-1} + x\big)\big(q^{-1}t+qt^{-1} + x\big)}\\
\qquad{} = F_2(t,x).
\end{gather*}
Concerning $F_3$,
\begin{gather*}
 -\sum_{n=1}^\infty (-1)^n t^n \sum_{\ell \in \mathbb N}\binom{\ell+1}{2} \lbrack n-2\ell-1\rbrack_q U_{n-2\ell-2}(x)\\
\qquad{} = -\sum_{n=1}^\infty \sum_{\ell \in \mathbb N}(-1)^{n} t^{n}\lbrack n-2\ell-1\rbrack_q U_{n-2\ell-2}(x) \binom{\ell+1}{2}\\
\qquad{} = \sum_{n=1}^\infty \sum_{\ell \in \mathbb N}(-1)^{n-2\ell-1} t^{n-2\ell -1}\lbrack n-2\ell-1\rbrack_q U_{n-2\ell-2}(x) \binom{\ell+1}{2} t^{2\ell+1}\\
\qquad{} = \sum_{N \in \mathbb N} \sum_{\ell \in \mathbb N}(-1)^{N} t^{N}\lbrack N \rbrack_q U_{N-1}(x) \binom{\ell+1}{2} t^{2\ell+1} \qquad \mbox{\rm (change var. $N=n-2\ell-1$)}\\
\qquad{} = \Biggl( \sum_{N \in \mathbb N}(-1)^{N} t^{N}\lbrack N \rbrack_q U_{N-1}(x)\Biggr)\Biggl(\sum_{\ell \in \mathbb N}\binom{\ell+1}{2} t^{2\ell+1} \Biggr)\\
\qquad{} = \frac{t-t^{-1}}{\big(qt+q^{-1}t^{-1} + x\big)\big(q^{-1}t+qt^{-1} + x\big)}\frac{-1}{\big(t-t^{-1}\big)^3}\\
\qquad{} = \frac{-1}{\big(t-t^{-1}\big)^2 \big(qt+q^{-1}t^{-1} + x\big)\big(q^{-1}t+qt^{-1} + x\big)}\\
\qquad{} = F_3(t,x).
\end{gather*}
Concerning $F_4$,
\begin{gather}
 \sum_{n=1}^\infty (-1)^n t^n \lbrack n\rbrack_q U_{n-1}(x)\nonumber\\
 \qquad {} =\sum_{n\in \mathbb N} (-1)^n t^n \lbrack n\rbrack_q U_{n-1}(x)\nonumber\\
\qquad{} = \frac{t-t^{-1}}{\big(qt+q^{-1}t^{-1} + x\big)\big(q^{-1}t+qt^{-1} + x\big)}\label{eq:n1}
\end{gather}
and also
\begin{gather}
 \sum_{n=1}^\infty (-1)^n t^n \sum_{\ell \in \mathbb N}\lbrack n-2\ell-2\rbrack_q U_{n-2\ell-3}(x)\nonumber\\
\qquad{} = \sum_{n=1}^\infty \sum_{\ell \in \mathbb N}(-1)^{n} t^{n}\lbrack n-2\ell-2\rbrack_q U_{n-2\ell-3}(x)\nonumber\\
\qquad{} = \sum_{n=1}^\infty \sum_{\ell \in \mathbb N}(-1)^{n-2\ell-2} t^{n-2\ell-2}\lbrack n-2\ell-2\rbrack_q U_{n-2\ell-3}(x) t^{2\ell+2} \nonumber\\
\qquad{} = \sum_{N \in \mathbb N} \sum_{\ell \in \mathbb N}(-1)^{N} t^{N}\lbrack N \rbrack_q U_{N-1}(x) t^{2\ell+2} \qquad {\mbox{\rm (change var. $N=n-2\ell-2$)}}\nonumber\\
\qquad{} = \Biggl( \sum_{N \in \mathbb N}(-1)^{N} t^{N}\lbrack N \rbrack_q U_{N-1}(x)\Biggr)\Biggl(\sum_{\ell \in \mathbb N} t^{2\ell+2} \Biggr)\nonumber\\
\qquad{} = -\frac{t-t^{-1}}{\big(qt+q^{-1}t^{-1} + x\big)\big(q^{-1}t+qt^{-1} + x\big)}\frac{ t}{t-t^{-1}}\nonumber\\
\qquad{} = -\frac{t}{\big(qt+q^{-1}t^{-1} + x\big)\big(q^{-1}t+qt^{-1} + x\big)}.\label{eq:n2}
\end{gather}
Note that (\ref{eq:n1}) plus twice (\ref{eq:n2}) is equal to $F_4(t,x)$.

Concerning $F_5$,
\begin{gather}
 \sum_{n=1}^\infty (-1)^n t^n \lbrack n+1\rbrack_q U_{n}(x)\nonumber\\
\qquad{} =-1+\sum_{n\in \mathbb N} (-1)^{n} t^{n} \lbrack n+1\rbrack_q U_{n}(x)\nonumber\\
\qquad{} =-1-t^{-1}\sum_{n\in \mathbb N} (-1)^{n+1} t^{n+1} \lbrack n+1\rbrack_q U_{n}(x)\nonumber\\
\qquad{} = -1-t^{-1} \sum_{N \in \mathbb N}(-1)^{N} t^{N}\lbrack N \rbrack_q U_{N-1}(x) \qquad {\mbox{\rm (change var. $N=n+1$)}}\nonumber\\
\qquad{} = -1 - \frac{t^{-1}\big(t-t^{-1}\big)}{\big(qt+q^{-1}t^{-1} + x\big)\big(q^{-1}t+qt^{-1} + x\big)}\label{eq:m1}
\end{gather}
and also
\begin{gather}
 \sum_{n=1}^\infty (-1)^n t^n \lbrack n-1\rbrack_q U_{n-2}(x)\nonumber\\
\qquad{} =-t\sum_{n=1}^{\infty} (-1)^{n-1} t^{n-1} \lbrack n-1\rbrack_q U_{n-2}(x)\nonumber\\
\qquad{} = -t \sum_{N \in \mathbb N}(-1)^{N} t^{N}\lbrack N \rbrack_q U_{N-1}(x) \qquad {\mbox{\rm (change var. $N=n-1$)}}\nonumber\\
\qquad{} = -\frac{t\big(t-t^{-1}\big)}{\big(qt+q^{-1}t^{-1} + x\big)\big(q^{-1}t+qt^{-1} + x\big)},\label{eq:m2}
\end{gather}
and also
\begin{gather}
 \sum_{n=1}^\infty (-1)^n t^n \sum_{\ell \in \mathbb N}\lbrack n-2\ell-3\rbrack_q U_{n-2\ell-4}(x)\nonumber\\
\qquad{} = \sum_{n=1}^\infty \sum_{\ell \in \mathbb N}(-1)^{n} t^{n}\lbrack n-2\ell-3\rbrack_qU_{n-2\ell-4}(x)\nonumber\\
\qquad{} = - \sum_{n=1}^\infty \sum_{\ell \in \mathbb N}(-1)^{n-2\ell-3} t^{n-2\ell-3}\lbrack n-2\ell-3\rbrack_q U_{n-2\ell-4}(x)t^{2\ell+3}\nonumber\\
\qquad{} = - \sum_{N \in \mathbb N} \sum_{\ell \in \mathbb N}(-1)^{N} t^{N}\lbrack N \rbrack_q U_{N-1}(x) t^{2\ell+3} \qquad {\mbox{\rm (change var. $N=n-2\ell-3$)}}\nonumber\\
\qquad{} = - \Biggl( \sum_{N \in \mathbb N}(-1)^{N} t^{N}\lbrack N \rbrack_q U_{N-1}(x)\Biggr)\Biggl(\sum_{\ell \in \mathbb N} t^{2\ell+3} \Biggr)\nonumber\\
\qquad{} = \frac{t-t^{-1}}{\big(qt+q^{-1}t^{-1} + x\big)\big(q^{-1}t+qt^{-1} + x\big)}\frac{t^2}{t-t^{-1}}\nonumber\\
\qquad{} = \frac{t^2}{\big(qt+q^{-1}t^{-1} + x\big)\big(q^{-1}t+qt^{-1} + x\big)}.\label{eq:m3}
\end{gather}
Note that $(-1)$ times (\ref{eq:m1}) minus $\lbrack 3 \rbrack_q$ times (\ref{eq:m2}) minus $\lbrack 2 \rbrack_q^2$ times (\ref{eq:m3}) is equal to $F_5(t,x)$. The result follows from the above comments.}
\end{proof}

Recall the center $Z(\Delta_q)$.
\begin{Corollary} For $n\geq 1$ the element $B_{n\delta}$ is contained in the subalgebra of $\Delta_q$ generated by $C$ and $Z(\Delta_q)$.
\end{Corollary}

We finish the paper with some comments.

Here is another version of Theorem~\ref{thm:main1short}.

\begin{Proposition} \label{thm:main1alt} For $n\geq 0$ the following hold in~$\Delta_q$:
\begin{gather*}\allowdisplaybreaks
B_{n\delta+\alpha_0} = (-1)^n q^{n} U_n(C) A+(-1)^n q^{n+1} U_{n-1}(C) B+(-1)^n \alpha \sum_{j \in \mathbb N} q^{n-2j-1} U_{n-2j-2}(C)\\
\hphantom{B_{n\delta+\alpha_0} =}{} + (-1)^{n-1} \beta\sum_{j\in \mathbb N} q^{n-2j} U_{n-2j-1}(C),\\
B_{n\delta+\alpha_1} = (-1)^n q^{-n} U_n(C) B+ (-1)^n q^{-n-1} U_{n-1}(C) A+(-1)^n \beta \sum_{j \in \mathbb N}q^{2j-n+1} U_{n-2j-2}(C)\\
\hphantom{B_{n\delta+\alpha_1} =}{} + (-1)^{n-1}\alpha\sum_{j\in \mathbb N} q^{2j-n} U_{n-2j-1}(C).
\end{gather*}
\end{Proposition}
\begin{proof} Similar to the proof of Theorem~\ref{thm:main1short}.
\end{proof}

The following result might be of independent interest.

\begin{Proposition} \label{prop:TCA}For $n\geq 1$ the following holds in $\Delta_q$:
\begin{gather*}
U_n(C) A = q^{-2n}A U_n(C)-q^2 \big(q-q^{-1}\big)A \sum_{\ell \in \mathbb N}\lbrack 2n-4\ell-2\rbrack_q U_{n-2\ell-2}(C)\\
\hphantom{U_n(C) A =}{} - q^{-1} \big(q-q^{-1}\big) B\sum_{\ell \in \mathbb N} \lbrack 2n-4\ell\rbrack_q U_{n-2\ell-1}(C)\\
\hphantom{U_n(C) A =}{}+ \big(q-q^{-1}\big)^2 \alpha \sum_{\ell \in \mathbb N} \lbrack n-2\ell-1\rbrack_q\lbrack \ell+1 \rbrack_q\lbrack {n-\ell}\rbrack_qU_{n-2\ell-2}(C)\\
\hphantom{U_n(C) A =}{}+ \big(q-q^{-1}\big)\beta \sum_{\ell \in \mathbb N} \lbrack n-2\ell\rbrack_q\big(q^{\ell-n}\lbrack \ell+1\rbrack_q-q^{n-\ell+1} \lbrack \ell \rbrack_q\big)U_{n-2\ell-1}(C)
\end{gather*}
and also
\begin{gather*}
U_n(C) B = q^{2n}B U_n(C) + q^{-2} \big(q-q^{-1}\big)B \sum_{\ell \in \mathbb N} \lbrack 2n-4\ell-2\rbrack_q U_{n-2\ell-2}(C)\\
\hphantom{U_n(C) B =}{} + q \big(q-q^{-1}\big) A\sum_{\ell \in \mathbb N} \lbrack 2n-4\ell\rbrack_q U_{n-2\ell-1}(C)\\
\hphantom{U_n(C) B =}{} + \big(q-q^{-1}\big)^2 \beta \sum_{\ell \in \mathbb N} \lbrack n-2\ell-1\rbrack_q \lbrack \ell+1 \rbrack_q \lbrack {n-\ell}\rbrack_q U_{n-2\ell-2}(C)\\
\hphantom{U_n(C) B =}{} - \big(q-q^{-1}\big)\alpha \sum_{\ell \in \mathbb N} \lbrack n-2\ell\rbrack_q \big(q^{n-\ell}\lbrack \ell+1\rbrack_q-q^{\ell-n-1} \lbrack \ell \rbrack_q\big)U_{n-2\ell-1}(C).
\end{gather*}
\end{Proposition}

\begin{proof} We use induction on $n$. For $n=1$ the equations in the proposition statement are reformulations of (\ref{eq:alpha}), (\ref{eq:beta}). For $n\geq 2$ we proceed as follows. To obtain the first (resp.\ second) equation in the proposition statement, multiply each side of~(\ref{eq:beta}) (resp.~(\ref{eq:alpha})) on the left by $U_{n-1}(C)$, and evaluate the result using $ CU_{n-1}(C) =U_n(C)+U_{n-2}(C)$ along with induction and Lemmas~\ref{lem:brRec},~\ref{lem:ident3}.
\end{proof}

In the algebra $\mathcal O_q$ the elements $\lbrace B_{n\delta}\rbrace_{n=1}^\infty $ are defined using the formula (\ref{eq:Bdelta}). This formula is not symmetric in $\alpha_0$, $\alpha_1$. As shown in \cite{BK}, there is another formula for $\lbrace B_{n\delta}\rbrace_{n=1}^\infty $ that interchanges the roles of $\alpha_0$, $\alpha_1$. According to \cite[Section~5.2]{BK} the following holds in $\mathcal O_q$ for $n\geq 1$:
\begin{gather} \label{eq:Bdel2}
B_{n \delta} = q^{-2} B B_{(n-1)\delta+\alpha_0}- B_{(n-1)\delta+\alpha_0} B +\big(q^{-2}-1\big)\sum_{\ell=0}^{n-2} B_{\ell \delta+\alpha_0}B_{(n-\ell-2) \delta+\alpha_0}.
\end{gather}
We now sketch a proof of Theorem~\ref{thm:v3} that uses~(\ref{eq:Bdel2}) instead of~(\ref{eq:Bdelta}). Following~(\ref{eq:genf}), for the algebra~$\Delta_q$ we define
\begin{gather}
\tilde \Phi(t) = \sum_{n=0}^\infty t^n B_{n\delta+ \alpha_0}. \label{eq:genf2}
\end{gather}
By (\ref{eq:genf}), (\ref{eq:Bdel2}), (\ref{eq:genf2}) we obtain
\begin{gather}
\Psi(t) = q^{-2} t B\tilde \Phi(t) - t \tilde \Phi(t)B + \big(q^{-2}-1\big) t^2 \bigl(\tilde \Phi(t)\bigr)^2. \label{eq:BdeltaGF2}
\end{gather}
By (\ref{eq:pre1}), (\ref{eq:pre2}),
\begin{gather}
\frac{\lbrack \tilde \Phi(t),C \rbrack}{q-q^{-1}} = t^{-1} A + B + \big(t-t^{-1}\big) \tilde \Phi(t). \label{eq:comGF2}
\end{gather}
From the first equation in Theorem~\ref{thm:main1short} we obtain
\begin{gather}
\tilde \Phi(t) \bigl(q^{-1}t+q t^{-1} + C\bigr) = q t^{-1} A - q^{-1} B - \frac{\alpha}{t-t^{-1}}- \frac{t^{-1} \beta}{t-t^{-1}}.\label{eq:PhiC2}
\end{gather}
By (\ref{eq:comGF2}) and (\ref{eq:PhiC2}),
\begin{gather}
\bigl(q t+q^{-1} t^{-1} + C \bigr)\tilde \Phi(t) =q^{-1} t^{-1} A - q B - \frac{\alpha}{t-t^{-1}} - \frac{t^{-1} \beta}{t-t^{-1}}.\label{eq:CPhi2}
\end{gather}
In~(\ref{eq:BdeltaGF2}), we multiply each side on the left by $qt + q^{-1}t^{-1} +C$ and on the right by $q^{-1} t + qt^{-1} +C$. We evaluate the result using~(\ref{eq:PhiC2}),~(\ref{eq:CPhi2}) to obtain
\begin{gather*}
\bigl(q t + q^{-1} t^{-1} + C\bigr)\Psi(t)\bigl(q^{-1} t + q t^{-1} + C\bigr)\\
\qquad{} = q^{-2} t \bigl(q t + q^{-1} t^{-1} + C\bigr)B\left(q t^{-1} A - q^{-1} B - \frac{\alpha}{t-t^{-1}}- \frac{t^{-1} \beta}{t-t^{-1}}\right)\\
\qquad\quad{} -t \left( q^{-1} t^{-1} A - q B - \frac{\alpha}{t-t^{-1}}- \frac{t^{-1} \beta}{t-t^{-1}}\right) B\bigl(q^{-1} t + q t^{-1} + C\bigr)\\
\qquad\quad{} + \big(q^{-2}-1\big)t^2 \left(q^{-1} t^{-1} A - q B - \frac{\alpha}{t-t^{-1}}- \frac{t^{-1} \beta}{t-t^{-1}}\right)\\
\qquad\quad{}\times \left(q t^{-1} A - q^{-1} B - \frac{\alpha}{t-t^{-1}}
- \frac{t^{-1} \beta}{t-t^{-1}}\right).
\end{gather*}
Evaluating this equation using
\begin{gather*}
BA = q^2 AB+q\big(q^2-q^{-2}\big)C - q\big(q-q^{-1}\big) \gamma,\\
CB = q^{2} BC+q\big(q^2-q^{-2}\big)A - q\big(q-q^{-1}\big) \alpha,\\
CB^2 =q^{4} B^2C+ q^3(q^4-q^{-4}) AB - q^{2} \big(q^2-q^{-2}\big)B\alpha\\
\hphantom{CB^2 =}{} +q^4 \big(q^2-q^{-2}\big)^2 C -q^4\big(q-q^{-1}\big)\big(q^2-q^{-2}\big)\gamma
\end{gather*}
and
\begin{gather*}
ABC = q^{-1} \Omega- q A^2-q^{-3} B^2- q C^2+ A \alpha+ q^{-2} B \beta+ C\gamma,\\
CBA = q \Omega- q^{-1} A^2- q^{3} B^2- q^{-1} C^2+ A \alpha+ q^{2} B \beta+ C\gamma
\end{gather*}
we find that
\begin{gather*}
\bigl(q t + q^{-1} t^{-1} + C\bigr) \Psi(t)\bigl(q^{-1} t + q t^{-1} + C\bigr)\end{gather*}
is equal to $1-q^{-2}$ times
\begin{gather*}
 \Omega - \frac{\big(t+t^{-1}\big) \alpha \beta }{\big(t-t^{-1}\big)^2}-\frac{\alpha^2+ \beta^2}{\big(t-t^{-1}\big)^2}- \big(t+t^{-1}\big)\gamma+\big(q+q^{-1}\big)\big(t+t^{-1}\big) C + C^2.
\end{gather*}
After this point, the present proof is the same as the original proof.

\subsection*{Acknowledgements}
The author thanks Pascal Baseilhac and Samuel Belliard for giving this paper a close reading and offering valuable suggestions.

\pdfbookmark[1]{References}{ref}
\LastPageEnding

\end{document}